\documentclass[11pt]{article}

\usepackage{amsmath}

\usepackage[ansinew]{inputenc}
\usepackage{amssymb}
\usepackage{amsthm}
\usepackage{amsfonts}
\usepackage{amscd}
\usepackage{latexsym, amsfonts}
\usepackage{graphicx}

\newtheorem{theorem}{Theorem}
\newtheorem{lemma}{Lemma}

\newtheorem{proposition}{Proposition}

\newtheorem{corollary}{Corollary}

\newcommand{\EDS}{\includegraphics{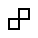}}
\newcommand{\EDB}{\includegraphics{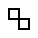}}
\newcommand{\EH}{\includegraphics{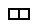}}
\newcommand{\EV}{\includegraphics{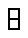}}

\newcommand{\be}{\begin{equation}}
\newcommand{\ee}{\end{equation}}
\newcommand{\rem}[1]{}

\def\A{\mathcal{A}}
\def\W{\mathcal{W}}
\def\B{\mathcal{B}}

\def\G{\mathcal{G}}
\def\S{\mathcal{S}}
\def\v{\mathcal{V}}
\def\h{\mathcal{H}}
\def\done{\mathcal{D}_1}
\def\dtwo{\mathcal{D}_2}
\def\cone{\mathcal{C}_1}
\def\ctwo{\mathcal{C}_2}
\def\linf{L_{\infty}}
\def\Gs{\mathcal{G}^{s}}
\def\t{\mathcal{T}}
\numberwithin{equation}{section}

\title{On totally disconnected generalised Sierpi\'nski carpets}

\author{Ligia L. Cristea $^*$ \\Institute of Mathematics\\Technical University of Graz\\Steyrergasse 30\\8010 Graz\\Austria\\ \tt{strublistea@gmail.com} \and Bertran Steinsky \thanks{This research was supported by the Austrian Science Fund (FWF), Project P20412-N18, at the Technical University of Graz, Institute of Mathematics A.} \\F\"urbergstr. 56,\\5020 Salzburg\\Austria\\ \tt{steinsky@finanz.math.tu-graz.ac.at}}

\begin{document}
\maketitle

\begin{abstract}
Generalised Sierpi\'nski carpets are planar sets that generalise the well-known Sierpi\'nski carpet and are defined by means of sequences of patterns. We study the structure of the sets at the $k$th iteration in the construction of the generalised carpet, for $k\ge 1$. Subsequently, we show that certain families of patterns provide total disconnectedness of the resulting generalised carpets. Moreover, analogous results hold even in a more general setting. Finally, we apply the obtained results in order to give an example of a totally disconnected generalised carpet with box-counting dimension $2$.
\end{abstract}
\noindent
AMS Classification [2010]{54H05, 28A80, 05C10} 

\section{Introduction}
{\em Sierpi\'nski carpets} are self-similar fractals in the plane that originate from the Sierpi\'nski carpet\cite{Mandelbrot1983, Sierpinski} and are constructed
in the following way: start with the unit square, divide it into $n \times n$ congruent smaller subsquares and cut out $m$ of them, corresponding to a given $n \times n$ pattern (called the generator of the Sierpi\'nski carpet). This construction step is repeated with all the remaining subsquares ad infinitum. The resulting object is a fractal of Hausdorff and box-counting dimension $\log(n^2 -m)/ \log(n)$, called a \emph{Sierpi\'nski carpet}\cite{the_pore_structure}. At each step of the iterative construction the corresponding squares are deleted, together with their boundary, and then the closure (with respect to the topology induced by the Euclidean metric in the plane) is taken. Sierpi\'nski carpets have been used, e.g., as models for porous materials \cite{the_pore_structure, modelling_porous_structures}.
The ``Cantor dust'' (see, e.g., Falconer\cite{Falconer1990}) is an example of a totally disconnected Sierpi\'nski carpet, for $n=4$ and $m=4$.

In the present paper, we study planar sets that generalise the Sierpi\'nski carpets mentioned before, namely 
\emph{generalised Sierpi\'nski carpets} (which we sometimes call, in short, \emph{generalised carpets}). A generalised Sierpi\'nski carpet differs from a Sierpi\'nski carpet defined as above slightly in the construction and in the following aspects. At step $k$ of the construction, for any $k\ge 1$, we apply an $m_k \times m_k$ pattern, where $m_k\ge 2$, for all $k \ge 1$, and, at any two steps $k_1 \ne k_2$ we may have distinct patterns, with $m_{k_1}\ne m_{k_2}$. 

In a very recent paper Cristea and Steinsky\cite{CristeaSteinsky_ConnectedGSC} presented necessary and sufficient conditions, under which generalised Sierpi\'nski carpets are connected, with respect to Euclidean topology.
Hata\cite{Hata-selfsimilar85} studied connectedness properties of self similar fractals, and Cristea\cite{netsets} studied connectedness properties of fractals, that are, under certain conditions, a special case of the generalised Sierpi\'nski carpets analysed in the present paper.

Now, we give a short outline of this paper.
In the next section we present the construction of generalised Sierpi\'nski carpets by means of {\em patterns}, and several definitions of notions that occur throughout this paper. 
In Section \ref{sec:families_of_patterns} we introduce six special families of patterns which occur in the considerations of the sections to follow.
Section \ref{sec:structure} consists of results about the effect of these patterns on the structure of generalised Sierpi\'nski carpets. We show a somewhat ``periodic'' structure of the (complements of) the sets at the $k$th iteration in the construction of generalised carpets, for $k\ge 1$. In Section \ref{sec:totally_disconnected} we prove our main result, which gives sufficient conditions for generalised Sierpi\'nski carpets to be totally disconnected. We remark that while the proofs in our paper on connected generalised Sierpi\'nski carpets\cite{CristeaSteinsky_ConnectedGSC} are mainly based on graph-theoretical arguments, in the present paper, shapes and distances play an important role.  
Finally, as an application, we give an example of a construction of a totally disconnected generalised Sierpi\'nski carpet with box-counting dimension $2$.
\section{Definitions and construction}\label{sec:definitions_and_construction}
Let $x,y,q\in [0,1]$ such that $Q=[x,x+q]\times [y,y+q]\subseteq [0,1]\times [0,1]$. Then for any point $(z_x,z_y)\in[0,1]\times [0,1]$ we define the function
\[
P_Q(z_x,z_y)=(q z_x+x,q z_y+y).
\]
Let $m\ge 1$. $S_{i,j}^{m}=\{(x,y)\mid \frac{i}{m}\le x \le \frac{i+1}{m} \mbox{ and } \frac{j}{m}\le y \le \frac{j+1}{m} \}$ and  
${\cal S}_m=\{S_{i,j}^{m}\mid 0\le i\le m-1 \mbox{ and } 0\le j\le m-1 \}$. 
We call any nonempty $\A \subseteq \S_m$ an $m$-\emph{pattern}. Let $\{{\cal A}_k\}_{k=1}^{\infty}$ be a sequence of non-empty patterns and $\{m_k\}_{k=1}^{\infty}$ be the corresponding \emph{width-sequence, i.e., for all $k\ge 1$ we have ${\cal A}_k\subseteq {\cal S}_{m_k}$.} We let ${\cal W}_1={\cal A}_{1}$, and call it the \emph{set of white squares of level $1$}. Then we define ${\cal B}_1={\cal S}_{m_1} \setminus {\cal W}_1$ as the \emph{set of black squares of level $1$}. For $n\ge 2$ we define the \emph{set of white squares of level $n$} by 
\[{\cal W}_n=\bigcup_{W\in {\cal A}_{n}, W_{n-1}\in {\cal W}_{n-1}}\{ P_{W_{n-1}}(W)\}.
\]
\rem{\rem{\noindent}} For a sequence of patterns $\{{\cal A}_k\}_{k=1}^{\infty}$ with width sequence $\{m_k\}_{k=1}^{\infty}$ we introduce the notation $m(i):=\prod_{k=1}^{i}m_k$. 
In all the considerations to follow we will assume $m_k \ge 2$, for all $k \ge 1$. We note that ${\cal W}_n \subset {\cal S}_{m(n)}$, and we define the \emph{set of black squares of level $n$}, ${\cal B}_n={\cal S}_{m(n)} \setminus {\cal W}_n$. For $n\ge 1$, we define $L_n=\bigcup_{W\in {\cal W}_n} W$. Therefore, $\{L_n\}_{n=1}^{\infty}$ is a monotonically decreasing sequence of compact sets. We write $L_{\infty}=\bigcap_{n=1}^{\infty}L_n$, for the \emph{limit set of the pattern sequence $\{{\cal A}_k\}_{k=1}^{\infty}$}. 

For any $0\le i \le m(k)-1$ we call $\cup_{j=0}^{m(k)-1} \{S^{m(k)}_{i,j}\}$ a \emph{column of level} $k$. Moreover, we call $\cup_{j=0}^{m(k)-1} \{S^{m(k)}_{0,j}\}$ the \emph{left column of level} $k$ (in short the \emph{left column of $\S_{m(k)}$}), and $\cup_{j=0}^{m(k)-1} \{S^{m(k)}_{m(k)-1,j}\}$ the \emph{right column of level} $k$ (in short the \emph{right column of $\S_{m(k)}$}). Analogously, for any $0\le j \le m(k)-1$ we call $\cup_{i=0}^{m(k)-1} \{S^{m(k)}_{i,j}\}$ a \emph{row of level} $k$. $\cup_{i=0}^{m(k)-1} \{S^{m(k)}_{i,0}\}$ is the \emph{bottom row of level} $k$ (in short the \emph{bottom row of} $\S_{m(k)}$) and $\cup_{i=0}^{m(k)-1} \{S^{m(k)}_{i,m(k)-1}\}$ is the \emph{top row of level} $k$ (in short the \emph{top row of} $\S_{m(k)}$). From the above definitions we obtain, for $k=1$ and $m(1)=m$, the left and the right column in $\S_m$, and, respectively, the bottom and the top row in $\S_m$.

A {\em graph} $\mathcal{G}$ is a pair $({V},{E})$, where ${V}={V}(\mathcal{G})$ is a finite set of vertices, and the set of edges ${E}={E}(\mathcal{G})$ is a subset of $\{\{u,v\} \mid u,v \in {V}, u\neq v\}$. We write $u \sim v$ if $\{u,v\}\in {E}(\mathcal{G})$ and sometimes we say $u$ is a \emph{neighbour} of $v$. The sequence of vertices $\{u_i\}_{i=0}^{n}$ is a \emph{path between $u_0$ and $u_n$} in a graph $\mathcal{G}\equiv ({V},{E})$, if $u_0,u_1,\ldots,u_n\in {V}$, $u_{i-1}\sim u_i$ for $1 \le i\le n$, and $u_i\neq u_j$ for $0\le i<j\le n$. A \emph{connected component} is an equivalence class of the relation, where two vertices are related if there is a path between them.
A path in a graph $\G = (V(\G), E(\G))$ from a vertex $x\in V(\G)$ to a vertex $y \in V(\G)$ is called \emph{minimal} if any proper subsequence of the sequence of vertices of 
the path is not a path from $x$ to $y$.

For $\W \subseteq \S _m$ we define $\mathcal{G}({\cal W})\equiv ({V}(\mathcal{G}({\cal W})),{E}(\mathcal{G}({\cal W})))$ to be the graph of ${\cal W}$, i.e., the graph whose vertices ${V}(\mathcal{G}({\cal W}))$ are the squares in ${\cal W}$, and whose edges ${E}(\mathcal{G}({\cal W}))$ are the unordered pairs of distinct white squares, that have a nonempty intersection. We call any path in $\G(\B_n)$ a \emph{black path of level} $n$. If $p=\{S_i\}_{i=1}^r$ is a path in $\G(\W_n)$ or $\G(\B_n)$ then we call $\Gamma(p):=\cup_{i=1}^r S_i$ \emph{the corridor of the path} $p$.
Let $S_{x_0,y_0}^{m(n)}$ and $S_{x_1,y_1}^{m(n)}$ be white squares of level $n$.
If $e=\{S_{x_0,y_0}^{m(n)},S_{x_1,y_1}^{m(n)}\}\in \mathcal{E}(\mathcal{G}({\cal W}_{n}))$ 
then $e$ can be of four different types. If $|x_0-x_1|=1$ and $y_0=y_1$ then we say $e$ is 
$\EH$. If $|y_0-y_1|=1$ and $x_0=x_1$ then we say $e$ is $\EV$. If $(x_0-x_1)(y_0-y_1)=-1$ 
then we say $e$ is $\EDB$. If $(x_0-x_1)(y_0-y_1)=1$ then we say $e$ is $\EDS$.
\section{Special families of $m$-patterns}
\label{sec:families_of_patterns}
For an $m$-pattern $\A$ we denote by ${\A}^c$ the set $\S_m \setminus \A.$ For any $\A \subseteq \S_m$ we define $\Gs (\A)= \left( V(\Gs (\A)), E(\Gs (\A))\right)$ to be the graph whose set of vertices $V(\Gs (\A)) $ consists of the squares $S_{i,j}^{m}$ that are elements of $\A$ and whose set of edges consists of unordered pairs of distinct squares that are elements of $\A$ and have a common side. In the following, we introduce several particular types of pattern.
An $m$-pattern $\A$ is \emph{of type 
$\v$ (``vertically cutting'')}, if $\G({\A}^c)$ 
contains a connected component $\G(K)$ 
that corresponds to a subset $K$ of $\A^{c}$, connects the top and the bottom row of $\mathcal{S}_m$,
and has the property that there exist indices $i_1, i_2$ such that
\[
i_1 \in \{i, S_{i,m-1}^{m} \in K \},~~i_2 \in \{i, S_{i,0}^{m} \in K \} \mbox{ and } i_2 \in \{i_1-1,i_1,i_1+1 \}
.
\] 
We also denote by $\v$ the family of all patterns of type $\v$. An $m$-pattern $\A$ is \emph{of type $\h$ (``horizontally cutting'')},
if $\G({\A}^c)$ contains a connected component $\G(K)$ that corresponds to a subset $K$ of ${\A}^c$, connects the left and the right column of $\S_m$, and has the property that there exist indices $j_1, j_2$ such that
\[ 
j_1 \in \{ j \in S_{0,j}^{m} \in K\},~~j_2 \in \{j, S_{m-1,j}^{m} \in K \} \mbox{ and } j_2 \in \{j_1 -1, j_1, j_1 + 1\}.
\]
\begin{figure}[hhhh]
\begin{center}
\includegraphics[scale=0.6]{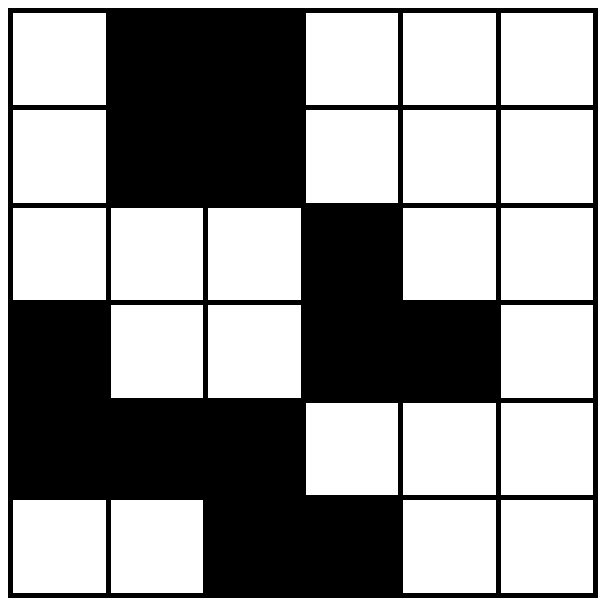}\hspace{0.2cm}\includegraphics[scale=0.6]{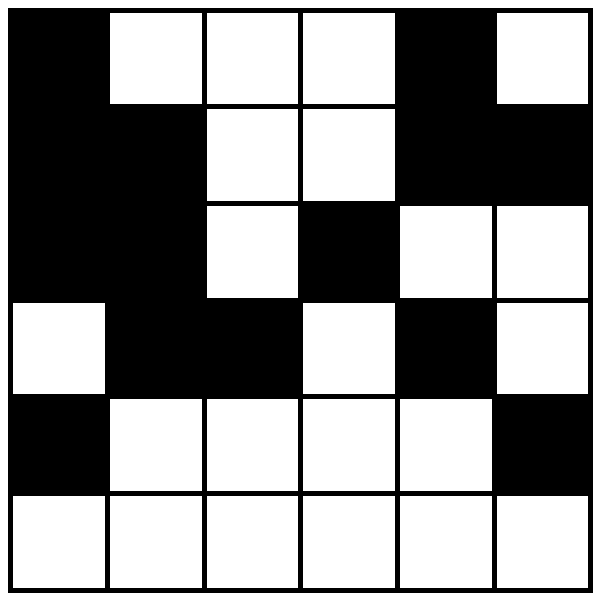}\hspace{0.2cm}\includegraphics[scale=0.6]{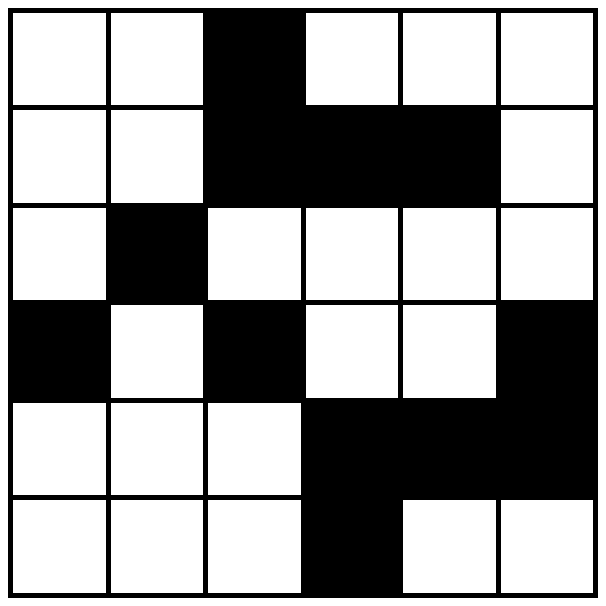}
\caption{Patterns of type $\v$, $\h$, and both $\v$ and $\h$, respectively.}\label{fig:vh}
\end{center}
\end{figure}
We also denote by $\h$ the family of all patterns of type $\h$. \rem{\noindent} An $m$-pattern $\A$ is \emph{of type
$\done$ (``diagonally cutting parallel to the first diagonal'')} in the following two cases:
\begin{enumerate}
\item
$\G({\A}^c )$ contains a connected component $\G(K)$ corresponding to a subset $K$ of ${\A}^c$, such that 
\[
\{S_{0,0}^{m},~S_{m-1,m-1}^{m}\} \subseteq K,
\]
\item 
$\G ({\A}^c )$ contains a connected component $\G(K_1)$ that corresponds to a subset $K_1$ of ${\A}^c$ and connects the left column and the top row of $\S_m$, and a connected component $\G(K_2)$ that corresponds to a subset $K_2$ of ${\A}^c$ and connects the bottom row and  the right column of $\S_m$, such that, on the one hand, there exist indices $j_1, j_2$ such that
\[ 
j_1\in \{ j, S_{0,j}^{m} \in K_1 \}, ~~
j_2 \in \{j, S_{m-1,j}^{m} \in K_2 \}, \mbox{ and } j_2 \in \{j_1-1, j_1, j_1 +1\}, 
\]
and, on the other hand, there exist indices $i_1, i_2$ such that
\[
i_1 \in \{i, S_{i,m-1}^{m} \in K_1 \},~~
i_2 \in \{i, S_{i,0}^{m} \in K_2 \}, \mbox{ and } i_2 \in \{i_1-1, i_1, i_1 +1\}.
\] 
\end{enumerate}
\rem{\noindent}
We also denote by $\done$ the family of all patterns of type $\done$. An $m$-pattern $\A$ is \emph{of type
$\dtwo$ (``diagonally cutting parallel to the second diagonal'')} in the following two cases:
\begin{enumerate}
\item 
$\G ({\A}^c)$ contains a connected component $\G(K)$ corresponding to a subset $K$ of ${\A}^c$, such that 
\[
\{S_{0,m-1}^{m},~S_{m-1,0}^{m}\} \subseteq K,
\]
\item  
$\G ({\A}^c )$ contains a connected component $\G(K_1)$ that corresponds to a subset $K_1$ of ${\A}^c$ and connects the left column and the bottom row of $\S_m$, and a connected component $\G(K_2)$ that corresponds to a subset $K_2$ of ${\A}^c$ and connects the top row and the right column of $\S_m$, such that, on the one hand, there exist indices $j_1, j_2$ such that
\[
 j_1 \in \{j, S_{0,j}^{m} \in K_1 \}, j_2 \in \{j, S_{m-1,j}^{m} \in K_2 \} \mbox{ and } j_2 \in \{j_1-1, j_1, j_1 +1\}, 
\] 
and, on the other hand, there exist indices $i_1, i_2$ such that
\[
i_1 \in \{i, S_{i,0}^{m} \in K_1 \}, i_2 \in \{i, S_{i,m-1}^{m} \in K_2 \} \mbox{ and } i_2 \in \{i_1-1, i_1, i_1 +1\}. 
\]
\end{enumerate} 
\rem{\noindent} We also denote by $\dtwo$ the family of all patterns of type $\dtwo$.

\begin{figure}[hhhh]
\begin{center}
\includegraphics[scale=0.45]{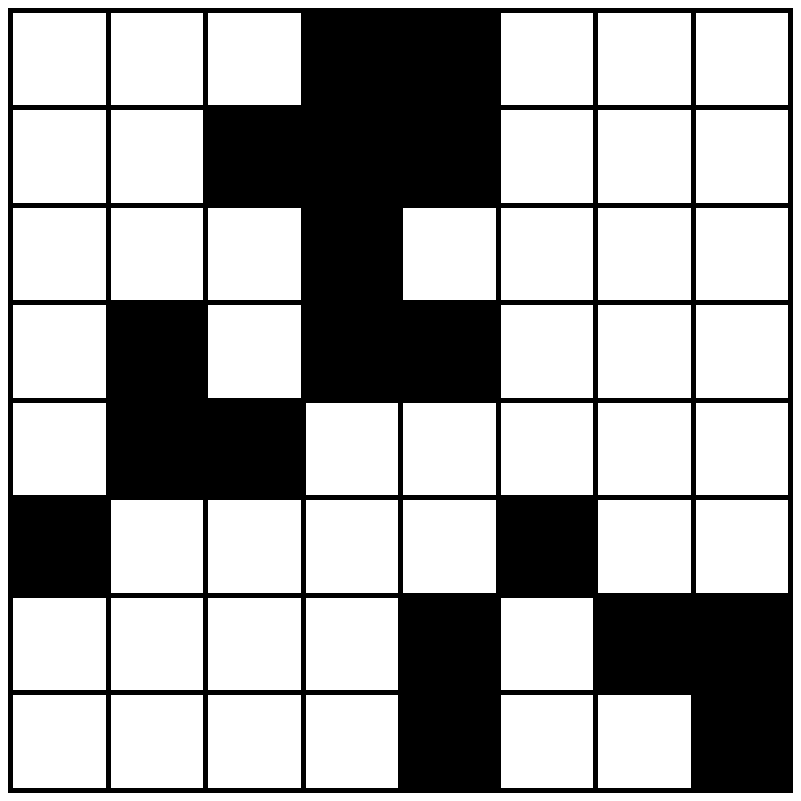}\hspace{0.2cm}\includegraphics[scale=0.45]{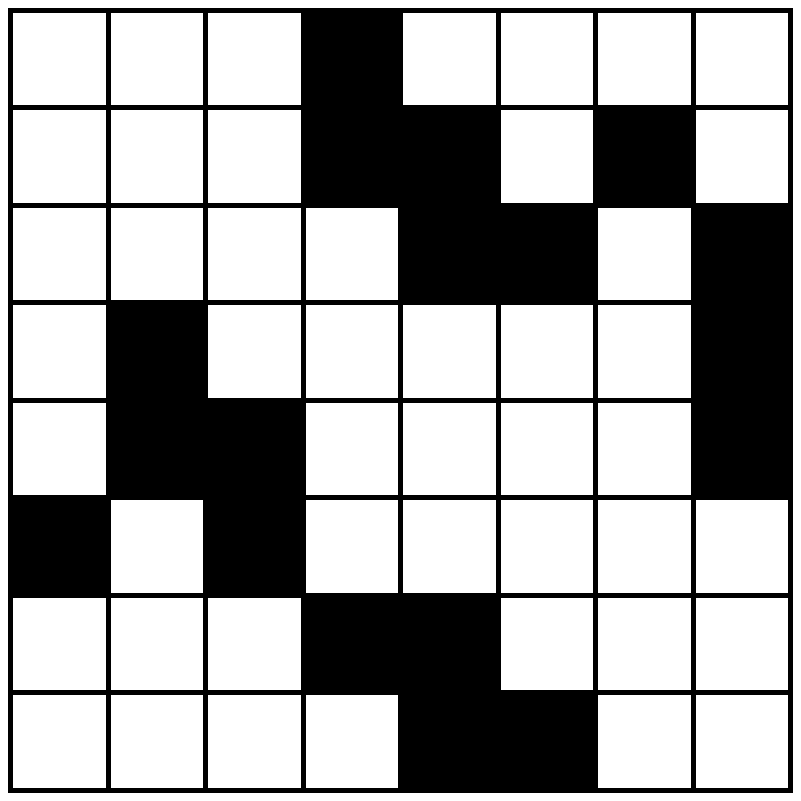}\hspace{0.2cm}\includegraphics[scale=0.45]{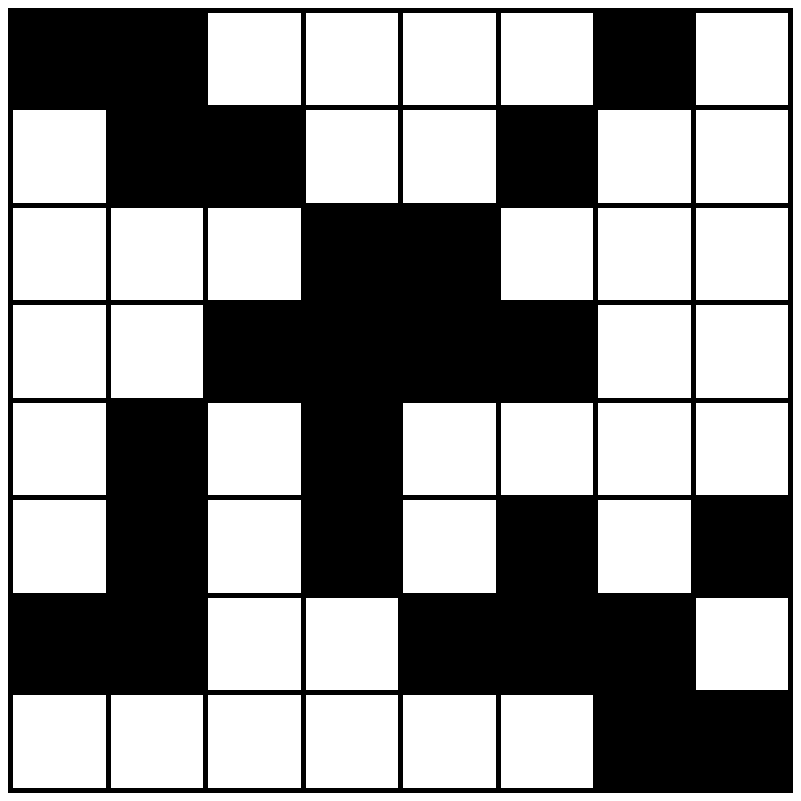}
\caption{Patterns of type $\done$, $\dtwo$, and both $\done$ and $\dtwo$, respectively. }\label{fig:d1d2}
\end{center}
\end{figure}
\rem{\noindent}
An $m$-pattern $\A$ is \emph{of type $\cone $, (``corner square on the first diagonal'')} if 
\[
 \{S_{0,0}^{m},S_{m-1,m-1}^{m}\} \cap {{\A}^c} \ne \emptyset,
\]
and \emph{of type $\ctwo $, (``corner square on the second diagonal'')} if
\[
\{S_{0,m-1}^{m},S_{m-1,0}^{m}\} \cap {{\A}^c} \ne \emptyset.
\]
We denote by $\cone$ and $\ctwo$ the family of all patterns of type $\cone$ and $\ctwo$, respectively.

\begin{figure}[hhhh]
\begin{center}
\includegraphics[scale=0.5]{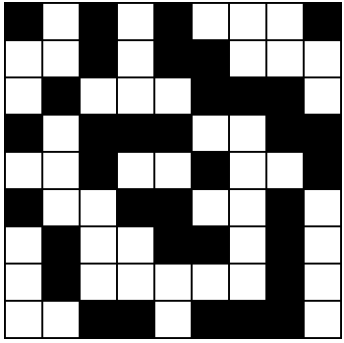}
\caption{A pattern of all types.}\label{fig:olle}
\end{center}
\end{figure}

\section{On the structure of generalised Sierpi\'nski carpets given by the occurrence of patterns of type $\v$, $\h$, $\done$, and $\dtwo$ }
\label{sec:structure}

Throughout this section, we assume, when dealing with sequences of patterns $\{{\cal A}_k\}_{k=1}^{\infty}$ that these patterns define generalised Sierpi\'nski carpets, i.e., we use ${\cal B}_k$, ${\cal W}_k$, and $L_k$ with the meaning given in Section~\ref{sec:definitions_and_construction}.

\begin{proposition} 
\label{prop_v_or_h_paths}
Let $\{\A_k\}_{k=1}^{\infty}$ be a sequence of patterns with width-sequence $\{m_k\}_{k=1}^{\infty}$. Let $1< k_1<k_2, k_3$, and $\A_{k_1} \in \v \cup \h$, $\A_{k_2} \in \cone$ and $\A_{k_3}\in \ctwo$. 
\begin{enumerate}
\item
If $\A_{k_1} \in \v$, then there exist $m(k_1 -1)$ distinct paths in $\Gs(\B_{k_3})$, each of them connecting some square of $\B_{k_3}$ lying in the top row of $\S_{m(k_3)}$ with some square of $\B_{k_3}$ lying in the bottom row of $\S_{m(k_3)}$. Each of these paths is contained in a column of level $k_1-1$. 
\item
If $\A_{k_1} \in \h$, the analogous statements hold for paths in $\Gs(\B_{k_3})$ that connect squares  that lie in the left column and in the right column of $\S_{m(k_3)}$.
\end{enumerate}
\end{proposition}
\begin{proof} 
We prove the statements for $\A_{k_1} \in \v$, as the proof for $\A_{k_1} \in \h$ is analogous. Also, we assume, without loss of generality, that $k_2\le k_3$.
Let us consider $Q_1, Q_2 \in V(\G(\W_{k_1 -1}))$ such that $Q_1$ and $Q_2$ lie in the same column of level $k_1 -1$ and have a common side, namely the bottom side of $Q_1$ and the top side of $Q_2$, respectively. As $\A_{k_1} \in \v$, it follows, by the definition of $\v$, that there exist squares $S_1, S_2 \in V(\G(\B_{k_1}))$ such that $S_1$ lies in the bottom row of level $k_1$ of $Q_1$,  $S_2$ lies in the top row of level $k_1$ of $Q_2$ and there exists an edge in $E(\G(\B_{k_1}))$ that connects $S_1$ and $S_2$. Thus, there is a path in $\G(\B_{k_1})$ contained in $Q_1 \cup Q_2$ that connects a black square of level $k_1$ in the top row of level $k_1$ of $Q_1$ with a black square of level $k_1$ in the bottom row of level $k_1$ of $Q_2$. Of course, this statement also holds if $Q_1 \in V(\G(\B_{k_1 -1}))$ or $Q_2 \in V(\G(\B_{k_1 -1}))$. Let us assume, e.g., $Q_1 \in V(\G(\B_{k_1 -1}))$. Then we can choose inside $Q_1$ the set of black squares that correspond to the set $K=K(\A_{k_1})$ occurring in the definition of $\A_{k_1} \in \v$. Furthermore, we can consider in $Q_1$ the black path of level $k_1$ that is the translation of the black path of level $k_1$ inside any other $Q \in V(\G(\W_{k_1 -1}))$, as the pattern $\A_{k_1}$ is applied to all white squares of level $k_1 -1$.

Let us now consider a column of level $k_1 -1$ in the unit square, denoted by $C$. We start from the square of level $k_1 -1$ situated in $C$ in the top row, and construct, as described above for $Q_1$ and $Q_2$ (passing to the square situated below $Q$ and then continuing inductively) a path in $\G(\B_{k_1})$ that connects a square lying in the top row of $\S_{m(k_1)}$ with a square lying in the bottom row of $\S_{m(k_1)}$. Let us denote this path by $p$. (By the construction of $p$, the corridor of $p$ is contained in the union of the squares of the column $C$.) 
The squares that are elements of $p$ correspond to the vertices in a connected component $\G(K_{k_1})$ of $\G(\B_{k_1})$. 

First, we consider the case when $\Gs(K_{k_1})$ is connected. Let $K_{k_3}$ be the set of all squares of level $k_3$ that are contained in some squares of level $k_1$ of $\G(K_{k_1})$. Then, with the above notations, we can chose inside the square $Q_1$, a path $p(Q_1)$ in $\Gs(\B_{k_3})$ that connects a square $T_1\in V(\G(\B_{k_3}))$ that lies in the top row of level $k_3$ of $Q_1$ with a square $R_1 \in V(\G(\B_{k_3}))$ that lies in the bottom row of level $k_3$ of $Q_1$, such that $T_1$ and $R_1$ lie in the same column of level $k_3$ of $Q_1$.
Now, by applying to the squares of $p(Q_1)$ the translation parallel to the $Oy$-axis that transforms $Q_1$ into $Q_2$ we obtain a path $p(Q_2)$ that connects a square $T_2$ situated in the top row of level $k_3$ of $Q_2$ with a square $R_2$ situated in the bottom row of level $k_3$ of $Q_2$, such that the bottom side of $R_1$ and the top side of $T_2$ coincide. Proceeding inductively we obtain a path $p'$ with $\Gamma(p')\subseteq \Gamma(p)$, by concatenating the paths $p(Q_1)$, $p(Q_2)$, \dots ,$p(Q_{k_1-1})$, where $Q_{k_1 -1}$ lies in the column $C$, in the bottom row of $\S_{m(k_1 -1)}$.

Now, we consider the case when $\Gs (K_{k_1})$ is not connected. Then there are two black squares contained in $K_{k_1}$ that are connected with each other only by 
$\EDS$ 
or 
$\EDB$ 
edges. Suppose $S'_1$ and $S'_2$ are two black squares of $K_{k_1}$ that are connected only by a 
$\EDB$ 
edge, where $S'_1$ is the upper square. Let $\{P\}= S'_1 \cap S'_2$. Then, at step $k_2 -1$ of the construction, $P$ is either the intersection point of exactly two black squares of level $k_2 -1$ (situated in the lower right corner of $S'_1$ and the upper left corner of $S'_2$) which are connected in $\G(\B_{k_2})$ by a 
$\EDB$ 
edge, or the intersection point of at least three black squares of level $k_2-1$. In both cases, applying the pattern $\A_{k_2} \in \cone$ on $L_{k_2 -1}$, yields, by the definition of $\cone$, that, at step $k_2$ of the construction of $\linf$, $P$ is the intersection of at least three distinct black squares of level $k_2$ belonging to $\B_{k_2}$. Analogously, applying the pattern $\A_{k_3}$ has the effect that points that were the intersection of exactly two black squares of level $k_2$, connected by a $\EDS$ edge, become the intersection of at least three black squares of level $k_3$. Let $B_{k}(p)$, for $k= k_2,k_3$, be the set of all black squares of level $k$ that lie in the same column $C$ and have a common side with some square of level $k_1$ that belongs to the path $p$. Let now $K_{k_3}$ be the set consisting of all black squares of level $k_3$ that are in the path $p$, together with all the black squares in $B_{k_3}(p)$, and all the black squares of level $k_3$ that are contained in some black square of $B_{k_2}(p)$ and share a side with some black square of level $k_3$  which is a subset of a black square occurring in $p$. Then $K_{k_3}$ contains a path $p'$ of level $k_3$ in $\G^s(\B_{k_3})$, with $\Gamma(p)\subseteq \Gamma(p')$, that connects a square of level $k_3$ lying in the top row of $\S_{m(k_3)}$ with some square of level $k_3$ lying in the bottom row of $\S_{m(k_3)}$, and $p'$ is contained in $C$.
\end{proof}
We call the paths occurring in Proposition~\ref{prop_v_or_h_paths} \emph{vertical paths of level} $k_3$ and \emph{horizontal paths of level} $k_3$ in the unit square, respectively.
Proceeding  analogously as in Proposition \ref{prop_v_or_h_paths}, one can prove the following result.
\begin{proposition}
\label{prop_diag_paths}
Let $\{\A_k\}_{k=1}^{\infty}$ be a sequence of patterns with width-sequence $\{m_k\}_{k=1}^{\infty}$. Let $1< k_1<k_2, k_3$, and $\A_{k_1} \in \done \cup \dtwo$, $\A_{k_2} \in \cone$ and $\A_{k_3}\in \ctwo$. 
\begin{enumerate}
\item If $\A_{k_1} \in \done$, then the following statements hold. There exist $m(k_1 -1)$ distinct paths in $\Gs(\B_{k_3})$, each of them connecting some square of $\B_{k_3}$ lying in the left column of $\S_{m(k_3)}$ with some square of $\B_{k_3}$ lying in the top row of $\S_{m(k_3)}$. There exist $m(k_1 -1)$ distinct paths in $\Gs(\B_{k_3})$, each of them connecting some square of $\B_{k_3}$ lying in the bottom row of $\S_{m(k_3)}$ with some square of $\B_{k_3}$ lying in the right column of $\S_{m(k_3)}$. 
\item
The analogous statements hold for $\A_{k_1} \in \dtwo$ and the corresponding paths in 
$\Gs(\B_{k_3})$. 
\end{enumerate}
\end{proposition}
\begin{corollary}(``Periodicity'' for vertical and horizontal patterns.) 
\label{cor_period_v_or_h}
Under the assumptions of Proposition \ref{prop_v_or_h_paths}, there is a set ${\cal P}$ of $m(k_1 -1)$ vertical paths of level $k_3$ in $\Gs(\B_ {k_3})$ such that the following statements hold.
\begin{enumerate}
\item 
If $p,p '\in {\cal P}$ and $Q, Q' \in \S_{m(k_1 -1)}$ lie in the same row (column) of level $k_1 -1$ such that $p$ passes through $Q$ and $p'$ passes through $Q'$, then there exists a translation $T$ by a vector of length $\frac{\alpha}{m(k_1 -1)}$, $\alpha \in \mathbb{N}$, parallel to the $Ox$- ($Oy$-) axis, such that $\Gamma (p') \cap Q' = T(\Gamma (p) \cap Q)$.
\item
If $p,p '\in {\cal P}$, then there exists a translation $T$  by a vector of length $\frac{\alpha}{m(k_1-1)}$  parallel to the $Ox$-axis, $\alpha   \in \mathbb{N}$, such that $\Gamma(p ' ) = T (\Gamma(p))$.
\end{enumerate}
The analogous statements hold for horizontal paths of level $k_3$ in the case of a horizontal pattern. 
\end{corollary}
\begin{proof}
Let ${\cal P}$ be the set of paths of level $k_3$ constructed for all columns of level $k_1-1$ as in the proof of Proposition \ref{prop_v_or_h_paths}. We can see that the intersection of the black squares of level $k_3$ belonging to a black vertical path of level $k_3$ in ${\cal P}$ with any square $Q$ of level $k_1-1$ is the translated image of the intersection of the black squares of level $k_3$ of any black path of level $k_3$ in ${\cal P}$ with any square $Q'$ of level $k_1 -1$. One can immediately see that $\frac{\alpha}{m(k_1-1)}$, with $1\le \alpha \le m(k_1 -1)-1$, is the distance between the left lower corners of any two distinct squares of level $k_1 -1$ that lie in the same row or column of level $k_1 -1$.
\end{proof}

\rem{\noindent} We call the paths occurring in Proposition~\ref{prop_diag_paths} \emph{diagonal paths of level} $k_3$ \emph{and type} $\done$, or, respectively, \emph{of type} $\dtwo$ in the unit square. Based on Proposition~\ref{prop_diag_paths} and the ideas of the above proof and the proof of Corollary~\ref{cor_period_v_or_h}, one can prove the following result.
\begin{corollary} (``Periodicity'' for diagonal patterns.)
\label{cor_period_diag}
Under the assumptions of Proposition \ref{prop_diag_paths}, there is a set ${\cal P}$ of diagonal paths of level $k_3$ and type $\done$ in $\Gs(\B_ {k_3})$ such that the following statements hold.
\begin{enumerate}
\item
If $p, p'\in {\cal P}$ and $Q, Q' \in \S_{m(k_1 -1)}$, then there exists a translation $T$ such that $\Gamma (p') \cap Q' = T(\Gamma (p) \cap Q)$.
\item
If $R \in V(\G(\B_{k_3}))$ lies on one of the diagonal paths of level $k_3$, then by translating $R$ parallel to the $Ox$-axis by a vector of length $\frac{\alpha}{m(k_1 -1)}$, $\alpha \in \mathbb{N}$,  to the left or to the right, we obtain a square $R'\in \B_{k_3}$ that either lies outside the unit square or on some other diagonal path.
\end{enumerate}
The analogous statements hold for diagonal paths of type $\dtwo$.
\end{corollary}
\begin{proposition} (``Parallel'' vertical curves for vertical patterns.)
\label{prop_paral_vert_curves}
Under the assumptions of Proposition \ref{prop_v_or_h_paths} let $\A_{k_1} \in \v$. Then there exists a set $\tilde{\v} (\A_{k_1})$  of curves that connect the top and the bottom side of the unit square with the following properties:
\begin{enumerate}
\item
If $\pi \in \tilde{\v} (\A_{k_1})$ and $Q, Q' \in \S_{m(k_1 -1)}$ lie in the same column of level $k_1 -1$ such that $\pi \cap Q \ne \emptyset$ and $\pi \cap Q' \ne \emptyset$,  then there exists a translation $T$ by a vector of length $\frac{\alpha}{m(k_1 -1)}$, $\alpha \in \mathbb{N},$ parallel to the $Oy$-axis, such that $\pi \cap Q' = T(\pi \cap Q)$.
\item
If $\pi, \pi ' \in \tilde{\v} (\A_{k_1})$ and $Q, Q' \in \S_{m(k_1 -1)}$ lie in the same row of level $k_1 -1$ such that $\pi \cap Q \ne \emptyset$ and $\pi \cap Q' \ne \emptyset$,  then there exists a translation $T$ by a vector of length $\frac{\alpha}{m(k_1 -1)}$, $\alpha \in \mathbb{N},$ parallel to the $Ox$-axis, such that $\pi \cap Q' = T(\pi \cap Q)$.
\item
If $\pi, \pi ' \in \tilde{\v} (\A_{k_1})$, then there exists a translation $T$  by a vector of length $\frac{\alpha}{m(k_1-1)}$, $\alpha   \in \mathbb{N}$,   parallel to the $Ox$-axis, such that $\pi ' = T (\pi)$.
\item
If $\pi \in \tilde{\v} (\A_{k_1})$, then it is contained in a column of level $k_1 -1$.  
\end{enumerate}
\end{proposition}

\begin{proof}
We consider, for each path in $\G^s(\B_{k_3})$ constructed in the proof of Proposition \ref{prop_v_or_h_paths} a minimal path $p^{min}$ from the top row to the bottom row of $\S_{m(k_3)}$, such that the $m(k_1 -1)$ minimal paths have the properties stated in Corollary \ref{cor_period_v_or_h}. Let now $p$ be such a path and $p^{min}$ the corresponding minimal sub-path. By definition, $p^{min}$ contains exactly one square $Q^{t} (p)$ of $V(\G^s(\B_{k_3}))$ that lies in the top row of level $k_3$ and exactly one square $Q^{b} (p)$ of $V(\G^s(\B_{k_3}))$ that lies in the bottom row of level $k_3$. Now we construct a curve inside the corridor of the path $p^{min}$. Let $p^{min}$ be the sequence $\{Q_r\}_{r=1}^{l}$ of squares of $V(\G^s(\B_{k_3}))$, such that $Q_1=Q^{t}(p)$ and $Q_{l}=Q^{b}(p)$. For $r= 1,\dots,l$ let $c(Q_r)$ be the centre of the square $Q_r$. Let $c^t(Q_1)$ be the midpoint of the top side of $Q_1$ and $c^b(Q_r)$ be the midpoint of the bottom side of $Q_r$. We construct a curve $\pi $ inside the corridor of $p^{min}$ by taking the union of the line segments $[c^t(Q_1), c(Q_1)]$, $[c(Q_1), c(Q_2)]$, $\dots$, $[c(Q_{r-1}), c(Q_{r})]$, $[c(Q_{r}), c^b (Q_r)]$. One can check, by applying Proposition \ref{prop_v_or_h_paths} and Corollary \ref{cor_period_v_or_h},
 that the family of curves, obtained by the above construction, has the properties stated in Proposition \ref{prop_paral_vert_curves}. 
\end{proof}
We note that the analogon of Proposition \ref{prop_paral_vert_curves} holds for patterns of type $\h$. With a construction idea analogous to that of the curves in the proof of Proposition
\ref{prop_paral_vert_curves} one can prove the following result. 
\begin{proposition} (``Parallel diagonal curves for diagonal patterns''.)
\label{prop_paral_diag_curves}
Under the assumptions of Proposition \ref{prop_diag_paths} let $\A_{k_1} \in \done$. Then there exists a set $\tilde{\done}(\A_{k_1})$ of curves that connect the left and the top side or the bottom and the right side of the unit square with the following properties:
\begin{enumerate}
\item
If $Q, Q' \in \S_{m(k_1-1)}$ lie in the same row (column) of level $k_1 -1$, then there exists a translation $T$ by a vector of length $\frac{\alpha}{m(k_1 -1)}$, $\alpha \in \mathbb{N},$ parallel to the $Ox$- ($Oy$)-axis, such that 
\[\{Q '\cap \pi '\mid \pi '\in \tilde{\done} (\A _{k_1}) \}= T \left( \{Q \cap \pi \mid \pi\in \tilde{\done} (\A _{k_1}) \} \right)  
\]
\item
If $\pi,\pi ' \in \tilde{\done}(\A_{k_1})$, then there exists a translation $T$  by a vector of length $\frac{\alpha}{m(k_1-1)}$, $\alpha \in \mathbb{N}$, parallel to the $Ox$-axis, such that either $\pi ' \subset T (\pi) $ or $T (\pi) \subset \pi '$.
\item
If $\pi,\pi ' \in \tilde{\done}(\A_{k_1})$, then there exists a translation $T$  by a vector of length $\frac{\alpha}{m(k_1-1)}$, $\alpha \in \mathbb{N}$, parallel to the $Oy$-axis, such that either $\pi ' \subset T (\pi)$ or $T(\pi) \subset \pi '$.
\end{enumerate}
\end{proposition}
\section{Totally disconnected generalised Sierpi\'nski carpets}
\label{sec:totally_disconnected}
\begin{lemma}
\label{lemma_tiles_v_and_h}
Let $\linf$ be a generalised carpet defined by a sequence of patterns $\{\A_k\}_{k=1}^{\infty}$ with width-sequence $\{m_k\}_{k=1}^{\infty}$. Let $1 < k_1 < k_2, k_3$,  and $k_1< k_4 < k_5, k_6$  such that 
\begin{enumerate}
\item
 $\A_{k_1} \in \v $, $\A_{k_2} \in \cone$ and $\A_{k_3} \in \ctwo$,
\item
 $\A_{k_4} \in \h $, $\A_{k_5} \in \cone$ and $\A_{k_6} \in \ctwo$.
\end{enumerate}
Then, for any two points $t=(t_1, t_2), z=(z_1, z_2)$ lying in the same connected component of $L_{k_6},$
\[ |t_1 - z_1| \le \frac{2}{m(k_1 -1)} \text{ and } |t_2 - z_2| \le \frac{2}{m(k_4 -1)}.
\]
\end{lemma}
\begin{proof}
Let $\Omega_{k_6}(t,z)$ be the connected component in $L_{k_6}$ that contains $t$ and $z$. We give a proof by contradiction. We assume that there is a column $C$ of level $k_1-1$ between $t$ and $z$. As $\Omega_{k_6}(t,z)$ is a finite union of squares, it is path-connected. Thus, there is a curve $c$ from $t$ to $z$ in $\Omega_{k_6}(t,z)$. Let $C'$ denote the rectangle that is the union of all squares of level $k_1 -1$ that belong to $C$. $c\cap C'$ is a curve from the left side of $C'$ to the right side of $C'$. 

From Proposition \ref{prop_paral_vert_curves} it follows that there exists a curve $\pi \in \tilde{\v}(\A_{k_6})$ such that $\pi$ is in $C'$ and leads from the top side of $C'$ to the bottom side of $C'$. We have $c\subseteq L_{\infty}$ and $\pi\subseteq[0,1]\times[0,1]\setminus L_{\infty}$, which is a contradiction to a known result, see e.g., Maehara\cite[Lemma~2]{Maehara}. We obtain that $t$ and $z$ must lie within two consecutive columns of level $k_1-1$, and therefore
$|t_1-z_1| \le \frac{2}{m(k_1 -1)}$. 
Using an analogon of Proposition \ref{prop_paral_vert_curves} for patterns of type $\h$ and the same arguments as before we infer $|t_2 - z_2| \le \frac{2}{m(k_4 -1)}$.
\end{proof}
\begin{lemma}
\label{lemma_tiles_diag_and_h}
Let $\linf$ be a generalised carpet defined by a sequence of patterns $\{\A_k\}_{k=1}^{\infty}$ with width-sequence $\{m_k\}_{k=1}^{\infty}$. Let $1 < k_1 < k_2, k_3$,  and $k_1< k_4 < k_5, k_6$  such that 
\begin{enumerate}
\item
 $\A_{k_1} \in \done \cup \dtwo $, $\A_{k_2} \in \cone$ and $\A_{k_3} \in \ctwo$,
\item
 $\A_{k_4} \in \h\cup\v $,$\A_{k_5} \in \cone$ and $\A_{k_6} \in \ctwo$.
\end{enumerate}
Then, for any two points $t=(t_1, t_2), z=(z_1, z_2)$ lying in the same connected component of $L_{k_6}$,
\[ |t_1 - z_1| \le \frac{4}{m(k_1 -1)} \text{ and } |t_2 - z_2| \le \frac{2}{m(k_1 -1)}.
\]
\end{lemma}
\begin{proof} We assume without loss of generality that $\A_{k_1} \in \done$ and $\A_{k_4} \in \h$. Let $\Omega_{k_6}(t,z)$ be the connected component in $L_{k_6}$ that contains $t$ and $z$. As $\Omega_{k_6}(t,z)$ is a finite union of squares, it is path-connected. Thus, there is a curve $c$ from $t$ to $z$ in $\Omega_{k_6}(t,z)$. Similarly as in Lemma~\ref{lemma_tiles_v_and_h}, we obtain that $c$ must lie within two consecutive rows $R_1$ and $R_2$ of level $k_1-1$, where $R_1$ is the upper row. Therefore, $|t_2-z_2| \le \frac{2}{m(k_1 -1)}$.

Now, we indirectly assume that between $t$ and $z$ there are at least $3$ columns of level $k_1-1$, whose union we denote by $C$. Let $Q=C\cap(R_1\cup R_2)$, and $Q'$ be the union of the squares of level $k_1-1$ in $Q$, i.e., $Q'$ is a rectangle. $c\cap Q'$ is a curve from the left side of $Q'$ to the right side of $Q'$.

We denote by $R_1 '$ and $R_2 '$ be the union of the squares of level $k_1 -1$ in $R_1$ and $R_2$, respectively.
As we know that $t$ and $z$ lie within $R_1$ and $R_2$, it follows from the definition of $\done$ and $\tilde \done$ that there is a curve $\pi$ in $Q'$ that leads from the top side of $R_1 '$ to the bottom side of $R_2 '$. Since $c\subseteq L_{\infty}$ and $\pi\subseteq[0,1]\times[0,1]\setminus L_{\infty}$, we have a contradiction to  Maehara\cite[Lemma~2]{Maehara}. Therefore, $t$ and $z$ must be contained in the union of at most $4$ consecutive columns of level $k_1-1$, which gives $|t_1-z_1| \le \frac{4}{m(k_1 -1)}$.
\end{proof}
\begin{lemma}
\label{lemma_tiles_diag_and_diag}
Let $\linf$ be a generalised carpet defined by a sequence of patterns $\{\A_k\}_{k=1}^{\infty}$ with width-sequence $\{m_k\}_{k=1}^{\infty}$. Let $1 < k_1 < k_2, k_3$  and $ k_4 < k_5, k_6$  such that 
\begin{enumerate}
\item
$\A_{k_1}\in \done $ and $\A_{k_4} \in \dtwo $, 
\item
$\A_{k_2} \in \cone$ and $\A_{k_3} \in \ctwo$,
\item
$\A_{k_5} \in \cone$ and $\A_{k_6} \in \ctwo$.
\end{enumerate}
Then, for any two points $t=(t_1, t_2), z=(z_1, z_2)$ lying in the same connected component of $L_{k_6}$,
\[ |t_1 - z_1| \le \frac{3}{m(k -1)} \text{ and } |t_2 - z_2| \le \frac{3}{m(k -1)}, ~~\text{ where }k=\min(k_1,k_4).
\]
\end{lemma}
\rem{\noindent} 
\begin{proof} As the first case, let either $t_1\le z_1$ and $t_2\ge z_2$ or $t_1\ge z_1$ and $t_2\le z_2$. If there are two columns between $t$ and $z$, then there is a diagonal block $B$ of level $k_1-1$ between $t$ and $z$. From the definition of $\done$ and $\tilde{\done}$ it follows that $B$ contains a curve $\pi \in \tilde{\v}(\A_{k_6})$, which leads from one edge of the unit square to another edge of the unit square. Therefore, $t$ and $z$ can not lie in the same connected component. The same holds if there are two rows between $t$ and $z$. Thus, we obtain, $|t_1 - z_1| \le \frac{3}{m(k_1 -1)}$ and $|t_2 - z_2| \le \frac{3}{m(k_1 -1)}$.

In the second case, where $t_1\ge z_1$ and $t_2\ge z_2$ or $t_1\le z_1$ and $t_2\le z_2$, we proceed analogously, and have $|t_1 - z_1| \le \frac{3}{m(k_4 -1)}$ and $|t_2 - z_2| \le \frac{3}{m(k_4 -1)}$.
\end{proof}
\rem{\noindent}
Before stating and proving the main result of this paper we remark that the construction of generalised Sierpi\'nski carpets, as it was given in Section \ref{sec:definitions_and_construction}, can be generalised, by
allowing, at each inductive step $k$ of the construction, not just the application of one pattern $\A_k \subset \S_{m_k}$ to all white squares that were created in the previous step, but, the application of a set of $n(k)$ distinct patterns $\{{\cal A}^i_k \}_{i=1}^{n(k)}$, $n(k)\ge 1$, $\A^i _k \subseteq \S_{m_k} $, with the possibility to apply  distinct patterns of $\{{\cal A}^i_k \}_{i=1}^{n(k)}$ to distinct white squares of $\W_{k-1}$. In this case we call $\linf$ a {\em non-uniform generalised Sierpi\'nski carpet}. Thus, a non-uniform generalised Sierpi\'nski carpet is defined by means of a sequence $\{\hat \A_k\}_{k=1}^{\infty}$, and its width sequence $\{m_k\}_{k=1}^{\infty}$,  where $\hat \A_k$ is a set of $n(k)$ (with $n(k)\ge 1$) $m_k$-patterns, for  all $k\ge 1$. 

The lemmas of this section also hold for non-uniform generalised Sierpi\'nski carpets, if we demand that for $i=1,\dots,6$, at step $k_i$ of the construction only one pattern is applied to all white squares of $\W_{k_{i}-1}$. In the same way, the results in Section \ref{sec:structure} can be extended to non-uniform generalised Sierpi\'nski carpets. 
\begin{theorem}
\label{theo_totally_disc}
Let $\linf$ be a generalised carpet defined by a sequence of patterns $\{\A_k\}_{k=1}^{\infty}$ with width-sequence $\{m_k\}_{k=1}^{\infty}$. If 
\begin{enumerate}
\item there exist two distinct types of patterns, $\t_1, \t_2 \in \{\v,\h , \done, \dtwo \}$
such that infinitely many patterns occurring in the sequence $\{\A_k\}_{k=1}^{\infty}$ are of type $\t_1$ and infinitely many patterns occurring in the sequence $\{\A_k\}_{k=1}^{\infty}$ are of type $\t_2$, and
\item infinitely many patterns occurring in the sequence $\{\A_k \}_{k=1}^{\infty}$ are of type $\cone$ and infinitely many patterns occurring in the sequence $\{\A_k \}_{k=1}^{\infty}$ are of type $\ctwo$, 
\end{enumerate}
then $\linf$ is totally disconnected with respect to the Euclidean topology.
\end{theorem}
\begin{proof} As the first case, we assume that $\t_1=\v$ and $\t_2=\h$. Lemma~\ref{lemma_tiles_v_and_h} yields that any connected component of $\linf$ consists of exactly one point. The second case is that $\t_1\in\{\done, \dtwo \}$ and $\t_2\in\{\h, \v\}$. Here, we use Lemma~\ref{lemma_tiles_diag_and_h} to obtain that any connected component of $\linf$ consists of one point. In the third and final case, we have $\t_1=\done$ and $\t_2=\dtwo$. By Lemma~\ref{lemma_tiles_diag_and_diag} we infer that any connected component of $\linf$ consists of one point.
\end{proof}
\rem{\noindent}
Based on the above proof, one can show that Theorem~\ref{theo_totally_disc} also holds in the case of non-uniform generalised Sierpi\'nski carpets:


\begin{theorem}
\label{theo_totally_disc_nonuniform}
Let $\linf$ be a non-uniform generalised carpet defined by a sequence of sets of patterns $\{\hat \A_k\}_{k=1}^{\infty}$ with width-sequence $\{m_k\}_{k=1}^{\infty}$. If 
\begin{enumerate}
\item there exist two distinct types of patterns, $\t_1, \t_2 \in \{\v,\h , \done, \dtwo \}$
such that infinitely many elements $\hat \A_k$ occurring in the sequence $\{\hat \A_k\}_{k=1}^{\infty}$ consist of only one pattern $\hat \A_k=\{\A_k\}$ and $\A_k \in \t_1$, and infinitely many elements occurring in the sequence $\{\hat \A_k\}_{k=1}^{\infty}$ consist of only one pattern $\hat \A_k=\{\A_k \}$ and $\A_k \in \t_2$, and 
\item infinitely many elements of the sequence $\{\hat \A_k \}_{k=1}^{\infty}$ satisfy
\[\hat \A_k = \{\A^i_{k}\}_{i=1}^{n(k)}, ~~n(k)\ge 1, \A^i_{k} \in \cone,~i=1,\dots,n(k),
\]
and 
infinitely many elements of the sequence $\{\hat \A_k \}_{k=1}^{\infty}$ satisfy
\[\hat \A_k = \{\A^i_{k}\}_{i=1}^{n(k)}, ~~n(k)\ge 1, \A^i_{k} \in \ctwo,~i=1,\dots,n(k),
\]
\end{enumerate}
then $\linf$ is totally disconnected with respect to the Euclidean topology.
\end{theorem}

\par \noindent {\bf Remark.} On the one hand, the results obtained here provide a method for constructing both self-similar and non-self-similar generalised carpets that are totally disconnected. On the other hand, the construction of the generalised Sierpi\'nski carpets described above makes it possible to obtain totally disconnected carpets of box-counting dimension less than or equal to $2$. 

For example, a generalised carpet $\linf$ of box-counting dimension $2$ is obtained if the sequence of patterns $\{\A_k\}_{k=1}^{\infty}$ with width-sequence $\{m_k\}_{k=1}^{\infty}$ that defines $\linf$ has the following properties:
\begin{enumerate} 
\item $m_n=n+1$, for all $n \ge 1$ (and thus $m(n)=(n+1)!$),
\item infinitely many patterns $\A_k$ are of type $\v$ and $\cone$, and the number of black squares in $\A_k$ is $k+1$,
\item the rest of infinitely many patterns $\A_k$ are of type $\done$ and $\ctwo$, and the number of black squares in $\A_k$ is $k+1$. 
\end{enumerate}
The above example is just one of various choices that are possible for the sequence of patterns $\{\A_k\}_{k=1}^{\infty}$ in order to obtain a totally disconnected generalised carpet with box-counting dimension $2$.
%

\end{document}